\documentclass[10pt]{amsart}
\usepackage{amssymb,latexsym,amsmath}
\usepackage{graphics}                 % Packages to allow inclusion of graphics
\usepackage{color}                    % For creating coloured text and background
\usepackage{hyperref}                 % For creating hyperlinks in cross references
\usepackage[all]{xy}
\begin{document}

\newtheorem{definition}{Definition}
\newtheorem{theorem}{Theorem}[section]
\newtheorem{proposition}[theorem]{Proposition}
\newtheorem{lemma}[theorem]{Lemma}
\newtheorem{corollary}[theorem]{Corollary}
\newtheorem{question}[theorem]{Question}
\newtheorem{remark}[theorem]{Remark}
\newtheorem{example}[theorem]{Example}
\newtheorem{conjecture}[theorem]{Conjecture}

\newtheorem{correction}{Correction}

\def\A{{\mathbb{A}}}
\def\C{{\mathbb{C}}}

\def\L{{\mathbb{L}}}
\def\O{{\mathcal{O}}}
\def\p{{\mathbf{p}}}

\title{Lawson homology groups of Chow varieties }
%\author{Youming Chen and Wenchuan Hu}

\date{\today}


%\address{
%School of Mathematics\\
%Sichuan University\\
%Chengdu 610064\\
%P.R. China}

%\email{huwenchuan@gmail.com,youmingchen@cqut.edu.cn}

\author[Y. Chen]{Youming Chen }
\address{School of Mathematical Science, Chongqing University of Technology, Chongqing 400054, P.R. China}%
\email{youmingchen@cqut.edu.cn}%

\author[W. Hu]{Wenchuan Hu}
\address{School of Mathematics, Sichuan University, Chengdu 610064, P.R. China}%
\email{huwenchuan@gmail.com}%

\begin{abstract}
Let $C_{p,d}(\mathbb{P}^n)$ denote the Chow variety of effective algebraic $p$-cycles of degree $d$ in complex projective space $\mathbb{P}^n$. In this paper, we compute the rational Lawson
homology groups $L_qH_k(C_{p,d}(\mathbb{P}^n))_\mathbb{Q}$ for $0 \leq 2q\leq k \leq 2d$.  Additionally, we prove that the rational Lawson homology groups of a natural completion of the Chow monoid of algebraic $p$-cycles in projective spaces are isomorphic to the corresponding rational singular homology groups. We also establish the stability of Lawson homology groups of Chow varieties under natural embeddings and algebraic suspension maps within a specified range.
\end{abstract}

\maketitle

\tableofcontents

\section{Introduction}

Let $\mathbb{P}^n$ denote the complex projective space of dimension $n$.
The Chow variety $C_{p,d}(\mathbb{P}^n)$ is the moduli space of effective algebraic $p$-cycles
of degree $d$ in $\mathbb{P}^n$.
Understanding its structure is crucial,
as it connects to fundamental problems in algebraic geometry, such as the Hodge Conjecture.

The Chow variety $C_{p,d}(\mathbb{P}^n)$ is central to the theory of algebraic cocycles
and the morphic cohomology developed by Friedlander and Lawson \cite{FL}.
Grothendieck \cite{Grothendieck}
and Friedlander--Mazur \cite[\S 7]{Friedlander-Mazur} conjectured
that the filtration on the singular cohomology of a smooth complex projective variety,
with rational coefficients,
induced by the cycle class map of morphic cohomology groups,
aligns with the Hodge filtration, corresponding to the maximal sub-Hodge structure.

Significant efforts over decades have focused on elucidating the structure of Chow varieties.
A foundational result by Chow and Van der Waerden establishes
that $C_{p,d}(\mathbb{P}^n)$ is a closed complex projective algebraic set
\cite{Chow-Waerden, GKZ, Kollar},
naturally equipping it with the structure of a compact Hausdorff space.
The topological and algebraic properties of $C_{p,d}(\mathbb{P}^n)$ are of particular interest.

The topological structure of the Chow variety $C_{p,d}(\mathbb{P}^n)$
becomes highly complex for large $p$, $n$, and $d$.
For instance, determining its dimension is nontrivial \cite{Eisenbud-Harris, Lehmann}.
While $C_{p,d}(\mathbb{P}^n)$ is easily seen to be a path-connected topological space,
proving its simple connectedness \cite{Horrocks, Fujiki, Lawson1} or computing topological invariants--such as higher homotopy groups,
(singular) homology groups \cite{Lawson1, Hu-Deg2, Hu-2015},
or Euler characteristic \cite{Lawson-Yau, Hu-2013} is challenging.
The complete characterization of higher homotopy or homology groups remains an open problem.
A concise survey of the structure of Chow varieties and related open questions is available
in \cite{Hu-2021} and its references.

The algebraic structure of the Chow variety $C_{p,d}(\mathbb{P}^n)$,
as a projective algebraic set, is notably complex.
Generally, $C_{p,d}(\mathbb{P}^n)$ is singular and non-irreducible,
with the number of irreducible components unknown for cases like $p=1$, $n=3$, and large $d$.
Shafarevich posed the question of whether each irreducible component of $C_{p,d}(\mathbb{P}^n)$
is rational \cite{Shafarevich},
to which a negative answer was provided \cite{Hu-2021}.
Recently we have computed Chow groups of $C_{p,d}(\mathbb{P}^n)$ in a certain range
(see \cite{Chen-Hu}).

The intricate algebraic structure of Chow varieties has limited progress in this area.
As a key algebraic invariant of an algebraic variety or scheme,
the Lawson homology groups are central to the study of algebraic cycles.
Notably, the Lawson homology group $L_1H_k(C_{p,d}(\mathbb{P}^n))$ for 1-cycles
has been shown to correspond to the associated homology group for all $d > 0$
and $0 \leq p \leq n$ \cite{Hu-2021}.

Furthermore, it has been wildly conjectured that
\begin{conjecture}[{\cite[Conj.~4.9]{Hu-2021}}]\label{conj1}
For all $d\geq 0$ and $0\leq p\leq n$,
$$
L_qH_k(C_{p,d}(\mathbb{P}^n))\cong H_{k}(C_{p,d}(\mathbb{P}^n),\mathbb{Z}),\ \forall k\geq 2q\geq 0,
$$
where $H_{*}(-,\mathbb{Z})$ denotes the singular homology group in integer coefficients.
\end{conjecture}

The Chow variety $C_{0,d}(\mathbb{P}^n)$
is the $d$-th symmetric product $\mathrm{SP}^d(\mathbb{P}^n)$ of $\mathbb{P}^n$.
For $p=0$, verifying Conjecture \ref{conj1} in this special case is crucial
for understanding the structure of symmetric products of projective spaces.
Given its significance, we formulate it as a separate conjecture.

\begin{conjecture}\label{conj2}
For all $d\geq 0$, $q\geq 0$ and $0\leq p\leq n$, the cycle class map
\begin{equation}\label{eqn1.1}
cl:L_qH_k(\mathrm{SP}^d(\mathbb{P}^n))\stackrel{\cong}{\longrightarrow} H_{k}(\mathrm{SP}^d(\mathbb{P}^n),\mathbb{Z})
\end{equation}
is an isomorphism.
\end{conjecture}

It may be overly optimistic to expect Conjecture \ref{conj1} to hold for integer coefficients,
even for the Chow variety parameterizing zero-cycles.
However, no evidence suggests that the cycle class map $cl$ is neither injective nor surjective.
In contrast, Conjecture \ref{conj2} is more likely to hold for $0 \leq q \leq d$.
Moreover, Equation \eqref{eqn1.1} can be shown
to hold after tensoring with rational coefficients (see Lemma \ref{lemma3.6}).
Consequently, we propose a weaker version of Conjecture \ref{conj1} as follows.

\begin{conjecture}\label{conj3}
For all $d\geq 0$ and $0\leq p\leq n$,
$$
L_qH_k(C_{p,d}(\mathbb{P}^n))_{\mathbb{Q}}\cong H_{k}(C_{p,d}(\mathbb{P}^n),\mathbb{Q}),\ \forall k\geq 2q\geq 0.
$$
%holds for all $ q\geq 0$, $d\geq 0$ and $n\geq 0$.
Here $H_{*}(-,\mathbb{Q})$ denotes the singular homology group in rational coefficients.
\end{conjecture}

This conjecture asserts that all even-dimensional topological cycles (in rational coefficients)
on Chow varieties are algebraic. Additionally, we conjecture
that all odd Betti numbers of $C_{p,d}(\mathbb{P}^n)$ vanish.

For a complex projective variety $X$,
the Lawson homology group of $q$-cycles is defined as
$$ L_qH_k(X) = \pi_{k-2q}(\mathcal{Z}_q(X)), $$
where $\mathcal{Z}_q(X)$ is the group of algebraic $q$-cycles on $X$,
endowed with a natural topology \cite{Lawson1, Friedlander1, Lawson2}.
Define $L_qH_k(X)_{\mathbb{Q}} := L_qH_k(X) \otimes \mathbb{Q}$.
Natural cycle class maps $cl: L_qH_k(X) \to H_k(X, \mathbb{Z})$ and $cl_{\mathbb{Q}}:
L_qH_k(X)_{\mathbb{Q}} \to H_k(X, \mathbb{Q})$ exist.
For further details on Lawson homology groups, see \cite{Lawson1, Friedlander1}.

Lawson homology groups encode both algebraic and topological information about a variety,
but their computation is generally very challenging.
For most varieties, the structure of these groups remains poorly understood.
Major conjectures, such as the generalized Hodge conjecture and the Suslin conjecture,
aim to elucidate the nature of Lawson homology groups.

In this paper,
we establish that the Lawson homology groups $L_qH_k(C_{p,d}(\mathbb{P}^n))_{\mathbb{Q}}$
coincide with their corresponding singular homology groups
in rational coefficients for $0 \leq 2q \leq k \leq 2d$,
thereby verifying Conjecture \ref{conj3} within this range.
Specifically, we compute $L_qH_k(C_{p,d}(\mathbb{P}^n))_{\mathbb{Q}}$ for $0 \leq 2q \leq k \leq 2d$,
leveraging prior results on singular homology groups in this range \cite{Hu-2015}.
Additionally, we demonstrate the stability of Lawson homology groups of Chow varieties
under natural embeddings or the algebraic suspension map within a specific range.
Precise statements are provided in the next section.

\thanks{\emph{Acknowledgements}: }
This work is partially supported by the National Natural Science Foundation of China
(No. 12126309, 12126354),
and the Natural Science Foundation of Chongqing (No. CSTB2025NSCQ-GPX0948).

\section{Main results}

%Let $X:=\mathrm{SP}^d\mathbb{P}^n$ be the $d$-th symmetric product of the projective space of dimension $n$.

%\textbf{Question:}Is the Lawson homology group of $p$-cycles  $\Ch_p(X)$ is isomorphic to the singular homology group $H_{2p}(X,\mathbb{Z})$?

%We know it is true in rational coefficients.

%Let $X$ be the $d$-th symmetric product of the complex projective space of dimension $n$.

%Question: Is the Lawson homology group of $p$-cycles $Ch_p(X)$ is isomorphic to the singular homology group %$H_{2p}(X,\mathbb{Z})$ with integer coefficients for all $p, d, n$?

%We know the answer to the question is positive in rational coefficients.

This paper adopts notation primarily from Lawson's work \cite{Lawson1}.
Let $l_0 \subset \mathbb{P}^n$ be a fixed $p$-dimensional linear subspace.
For each $d \geq 1$, we define the analytic embedding
\begin{equation}\label{eqn2.1}
i: C_{p,d}(\mathbb{P}^n) \hookrightarrow C_{p,d+1}(\mathbb{P}^n),
\quad c \mapsto c + l_0.
\end{equation}
This sequence of embeddings yields the union
$$ \mathcal{C}_{p}(\mathbb{P}^n) = \lim_{d \to \infty} C_{p,d}(\mathbb{P}^n), $$
termed the Friedlander completion of $C_{p,d}(\mathbb{P}^n)$.
The topology on $\mathcal{C}_{p}(\mathbb{P}^n)$ is the weak topology with respect to
$\{C_{p,d}(\mathbb{P}^n)\}_{d=1}^{\infty}$,
where a set $C \subset \mathcal{C}_{p}(\mathbb{P}^n)$ is closed
if and only if $C \cap C_{p,d}(\mathbb{P}^n)$ is closed
in $C_{p,d}(\mathbb{P}^n)$ for all $d \geq 1$.
For further details, see \cite{Lawson1}.

We define the Lawson homology group $L_qH_k(\mathcal{C}_p(\mathbb{P}^n))$ of algebraic $q$-cycles
on $\mathcal{C}_p(\mathbb{P}^n)$ as the direct limit
$\lim\limits_{d\to \infty}L_qH_k(C_{p,d}(\mathbb{P}^n))$.
Meanwhile, we have
%$$\lim_{d\to \infty}L_qH_k(C_{p,d}(\mathbb{P}^n))_{\mathbb{Q}}$$
$$
\lim_{d\to \infty}H_k(C_{p,d}(\mathbb{P}^n),\mathbb{Q})=H_k(\lim_{d\to \infty}C_{p,d}(\mathbb{P}^n),\mathbb{Q})
=H_k(\mathcal{C}_{p}(\mathbb{P}^n),\mathbb{Q})
$$
for all $0\leq p\leq n$ and $q\geq 0$.

Our first main result states as follows.

\begin{theorem} \label{Thm2.1}
For $0\leq p\leq n$ and $q\geq 0$, we have isomorphisms
$$
L_qH_k(\mathcal{C}_p(\mathbb{P}^n))_{\mathbb{Q}}\cong H_{k}(\mathcal{C}_p(\mathbb{P}^n),\mathbb{Q}).
$$
\end{theorem}

This verifies Conjecture \ref{conj3} for the Friedlander completion of Chow varieties.
Theorem \ref{Thm2.1} establishes that every topological cycle
in rational coefficients on $C_{p,d}(\mathbb{P}^n)$ is homologous to
an algebraic cycle under the direct limit of embeddings defined in \eqref{eqn2.1}.
In particular,
$$
\lim_{d\to\infty}\pi_k(\mathcal{Z}_q(C_{p,d}(\mathbb{P}^n)))\otimes \mathbb{Q}=0
$$
and
$$\lim_{d\to\infty}H_k(\mathcal{Z}_q(C_{p,d}(\mathbb{P}^n)))\otimes \mathbb{Q}=0$$
for all odd integers $k$.

The following result generalizes the main theorem of \cite{Hu-2015},
conjectured by Lawson, establishing
that the inclusion $i: C_{p,d}(\mathbb{P}^n) \to C_{p,d+1}(\mathbb{P}^n)$ is $2d$-connected.

\begin{theorem}\label{Thm2.2}
For all $0\leq p\leq n$,
the embedding  $i:C_{p,d}(\mathbb{P}^n)\to C_{p,d+1}(\mathbb{P}^n)$ induces isomorphisms
$$
i_*:L_qH_k(C_{p,d}(\mathbb{P}^n))_{\mathbb{Q}}\cong L_qH_k(C_{p,d+1}(\mathbb{P}^n))_{\mathbb{Q}}
$$
for $0\leq 2q\leq k\leq 2d$.
\end{theorem}

Combining results from the above theorems and the homotopy type of $\mathcal{C}_p(\mathbb{P}^n)$,
we determine the structure of the $q$-th Lawson homology group $L_qH_k(C_{p,d}(\mathbb{P}^n))$
for $C_{p,d}(\mathbb{P}^n)$:
$$
L_qH_k(C_{p,d}(\mathbb{P}^n))_{\mathbb{Q}}\cong H_{k}(\mathcal{C}_p(\mathbb{P}^n),\mathbb{Q})
$$
for all  $0\leq q\leq k\leq 2d$, where
$$
H_k(\mathcal{C}_p(\mathbb{P}^n),\mathbb{Q})\cong H_k\bigg(\prod_{i=1}^{n-p} K(\mathbb{Z},2i),\mathbb{Q}\bigg),
$$
and $K(\mathbb{Z},m)$ is the Eilenberg--MacLane space
whose $m$-th homotopy group is $\mathbb{Z}$,
and all other homotopy groups are trivial.
The $\mathbb{Q}$-dimension of $H_k\left(\prod_{i=1}^{n-p} K(\mathbb{Z}, 2i), \mathbb{Q}\right)$
is computable via the K\"{u}nneth formula.

The following result generalizes the Lawson suspension theorem for homotopy groups to Lawson homology groups.

\begin{theorem}\label{Thm2.3}
For all $0\leq p\leq n$,
the suspension map $\Sigma:C_{p,d}(\mathbb{P}^n)\to C_{p+1,d}(\mathbb{P}^{n+1})$ induces
isomorphisms
$$\Sigma_*:L_qH_k(C_{p,d}(\mathbb{P}^n))_{\mathbb{Q}}\cong L_qH_k(C_{p+1,d}(\mathbb{P}^{n+1}))_{\mathbb{Q}}$$ for $0\leq 2q\leq
k\leq 2d$.
\end{theorem}

The suspension map $\Sigma$ is defined in the next section.
We prove these theorems
by extending Lawson's approach to studying the homotopy type of Chow monoids \cite{Lawson1}.
Unlike Lawson's method for homotopy groups, our proofs involve significant modifications,
as Lawson homology groups are not homotopy invariants.

As applications of Theorem \ref{Thm2.1} and Theorem \ref{Thm2.2},
along with the Lawson suspension theorem \cite[Thm.~3]{Lawson1} and \cite[Thm.~1]{Hu-2015},
we compute the Lawson homology groups of $C_{p,d}(\mathbb{P}^n)$ up to degree $d$,
partially verifying Conjecture \ref{conj3}.

 \begin{corollary}\label{cor2.4}
For all $0\leq p\leq n$, the cycle class map
$$
cl:L_qH_k(C_{p,d}(\mathbb{P}^n))_{\mathbb{Q}}\cong H_{k}(C_{p,d}(\mathbb{P}^n),\mathbb{Q})
$$
is an isomorphism for any $0\leq 2q\leq k\leq 2d$.
\end{corollary}

The first $2d+1$ homology groups of $C_{p,d}(\mathbb{P}^n)$ is given by the formula
$$H_k(C_{p,d}(\mathbb{P}^{n}))\cong  H_k(K(\mathbb{Z}, 2)\times\cdots\times K(\mathbb{Z},2(n-p)))$$
in \cite[Cor. 5]{Hu-2015} for $0\leq k\leq 2d$.
Consequently,
Corollary \ref{cor2.4} yields the rank of $L_qH_k(C_{p,d}(\mathbb{P}^n))$
for $0 \leq 2q \leq k \leq 2d$ and all $0 \leq p \leq n$.

 \begin{corollary}\label{cor2.5}
For all $0\leq p\leq n$, we have isomorphisms
$$
L_qH_k(C_{p,d}(\mathbb{P}^n))_{\mathbb{Q}}\cong H_k(K(\mathbb{Z}, 2)\times\cdots\times K(\mathbb{Z},2(n-p)),\mathbb{Q})
$$
for any $0\leq 2q\leq k\leq 2d$.

\end{corollary}

\begin{remark}%\label{rmk2.6}
%If we have the following isomorphisms
%\begin{equation*}
%cl:L_qH_k(\mathrm{SP}^d(\mathbb{P}^n))\cong H_k(\mathrm{SP}^d(\mathbb{P}^n),\mathbb{Z}),
%\end{equation*}
%for all $q\geq 0$,i.e,
If the Conjecture \ref{conj2} holds, or
its special case that
$$
cl:L_qH_k(\mathrm{SP}^d(\mathbb{P}^n))\stackrel{\cong}{\longrightarrow} H_{k}(\mathrm{SP}^d(\mathbb{P}^n),\mathbb{Z}),\  0\leq 2q\leq k\leq 2d
$$
holds,
then Corollaries \ref{cor2.4} and \ref{cor2.5} hold in the same range in integer coefficients.
\end{remark}

\iffalse
For $0$-cycles,
the Dold-Thom theorem claims that $L_0H_k(C_{p,d}(\mathbb{P}^n))\cong H_k(C_{p,d}(\mathbb{P}^n))$ for all $0\leq p\leq n$, $d>0$ and $k\geq 0$.

For $1$-cycles, it has been shown in \cite{Hu-2021} that $L_1H_k(C_{p,d}(\mathbb{P}^n))\cong H_k(C_{p,d}(\mathbb{P}^n))$ for all $0\leq p\leq n$, $d>0$ and $k\geq 2$.
\fi

\section{The technique of Suspension on Lawson homology groups}

We briefly review Lawson's approach to proving the Complex Suspension Theorem.
For a detailed exposition of the Suspension Theorem
over the complex numbers or any algebraically closed field,
along with background material, see \cite{Lawson1, Friedlander1, Lawson2}.
Building on Lawson's work on the homotopy type of algebraic cycle spaces,
we compute the Lawson homology groups of Chow varieties and their Friedlander completion.

Fix a hyperplane $\mathbb{P}^n \subset \mathbb{P}^{n+1}$
and a point $\mathbb{P}^0 \in \mathbb{P}^{n+1}- \mathbb{P}^n$.
For any closed algebraic subset $V \subset \mathbb{P}^n$,
the algebraic suspension of $V$ with vertex $\mathbb{P}^0$ (i.e., the cone over $V$) is defined as
$$
\Sigma V := \cup \{ l \mid l \text{ is a projective line through } \mathbb{P}^0 \text{ intersecting } V \}.
$$
The suspension map $\Sigma$ induces a morphism
$\Sigma: C_{p,d}(\mathbb{P}^n) \to C_{p+1,d}(\mathbb{P}^{n+1})$
that commutes with the embedding $i$ in Equation \eqref{eqn2.1}.
Consequently, it induces a map
$\Sigma: \mathcal{C}_p(\mathbb{P}^n) \to \mathcal{C}_{p+1}(\mathbb{P}^{n+1})$.
Furthermore, $\Sigma$ induces homomorphisms on Lawson homology groups:
$$
\Sigma_*: L_qH_k(C_{p,d}(\mathbb{P}^n)) \to L_qH_k(C_{p+1,d}(\mathbb{P}^{n+1}))
$$
and
$$
\Sigma_*: L_qH_k(\mathcal{C}_p(\mathbb{P}^n)) \to L_qH_k(\mathcal{C}_{p+1}(\mathbb{P}^{n+1}))
$$
for all $k \geq 2q \geq 0$.

\begin{theorem}\label{Thm3.1}
For all $k\geq 2q\geq 0$,
the induced homomorphisms
$$
\Sigma_*:L_qH_k(\mathcal{C}_{p}(\mathbb{P}^{n})) \to L_qH_k(\mathcal{C}_{p+1}(\mathbb{P}^{n+1}))
$$
are surjective, and they are isomorphic after tensoring with rational coefficients.
\end{theorem}

%This is a version of the Lawson Suspension Theorem on homotopy and homology groups.

We will prove this theorem at the end of this section, as it requires preliminary results.
Consider a holomorphic vector field on $\mathbb{P}^{n+1}$ vanishing at $\mathbb{P}^n$
and its polar point $\mathbb{P}^0$.
Choose homogeneous coordinates $[z_0 : z_1 : \dots : z_{n+1}]$
for $\mathbb{P}^{n+1}$ such that
$\mathbb{P}^n = (z_0 = 0)$ and $\mathbb{P}^0 = [1 : 0 : \dots : 0]$. For $t \in \mathbb{C}^*$,
define
$$
\phi_t([z_0 : z_1 : \dots : z_{n+1}]) = [t z_0 : z_1 : \dots : z_{n+1}].
$$
The map $\phi_t: \mathbb{P}^{n+1} \to \mathbb{P}^{n+1}$ is an automorphism for $t \in \mathbb{C}^*$,
inducing an automorphism
$$
\phi_t: C_{p+1,d}(\mathbb{P}^{n+1}) \to C_{p+1,d}(\mathbb{P}^{n+1}).
$$
Clearly, $\phi_1$ is the identity map.
Moreover, $\phi_t$ preserves the subspace $T_{p+1,d}(\mathbb{P}^{n+1})$
and acts as the identity on $\Sigma C_{p,d}(\mathbb{P}^n)$ for all $t \in \mathbb{C}^*$,
where
$$
T_{p+1,d}(\mathbb{P}^{n+1}):=\big\{c=\sum n_iV_i\in C_{p+1,d}(\mathbb{P}^{n+1})|\dim(V_i\cap\mathbb{P}^n)=p, ~\forall i\big \}
$$
for any non-negative integers $p$ and $d$.
When $d = 0$, $C_{p,d}(\mathbb{P}^n)$ is defined as the empty cycle.

\begin{proposition}\label{Prop3.2}
The set $T_{p+1,d}(\mathbb{P}^{n+1})$ is Zariski open in $C_{p+1,d}(\mathbb{P}^{n+1})$.
Moreover, $T_{p+1,d}(\mathbb{P}^{n+1})$ is homotopy equivalent to $C_{p,d}(\mathbb{P}^{n})$.
For each $c\in T_{p+1,d}(\mathbb{P}^{n+1})$, the limit
$$
\phi_{\infty}(c):=\lim_{t\to \infty}\phi_t(c)\in \Sigma C_{p,d}(\mathbb{P}^{n})
$$
exists. Moreover,
there is an algebraic map
$$
\Phi: T_{p+1,d}(\mathbb{P}^{n+1})\times\C\to T_{p+1,d}(\mathbb{P}^{n+1})
$$
such that $\Phi_{|T_{p+1,d}(\mathbb{P}^{n+1})\times \{1\}}=\phi_1$
and $\Phi_{|T_{p+1,d}(\mathbb{P}^{n+1})\times \{\infty\}}=\phi_{\infty}$,
where $\C=\C^*\cup \{\infty\}$.
This implies that $\phi_1(Z)=Z$ is algebraic equivalent to $\phi_{\infty}(Z)$
if $Z$ is an (effective) cycle on $C_{p+1,d}(\mathbb{P}^{n+1})$
supported in $ T_{p+1,d}(\mathbb{P}^{n+1})$.

Let $\Lambda$ be a compact topological space
and $f: \Lambda \to \mathcal{Z}_q(C_{p+1,d}(\mathbb{P}^{n+1}))$ a continuous map
such that the support of $f(\lambda)$ is contained in $T_{p+1,d}(\mathbb{P}^{n+1})$
for each $\lambda \in \Lambda$.
Then $f$ is homotopic to a map $g: \Lambda \to \mathcal{Z}_q(C_{p+1,d}(\mathbb{P}^{n+1}))$
with the support of $g(\lambda)$ contained in $\Sigma C_{p,d}(\mathbb{P}^n)$
for each $\lambda \in \Lambda$.
\end{proposition}
\begin{proof}

The core idea of this proposition originates
from Fulton \cite[\S 5]{Fulton} and \cite[Rem. 4.6]{Lawson1}.
It is shown in \cite[\S 4]{Lawson1} that $T_{p+1,d}(\mathbb{P}^{n+1})$
is Zariski open in $C_{p+1,d}(\mathbb{P}^{n+1})$
and that $\phi_{\infty}(c)$ exists for each $c \in T_{p+1,d}(\mathbb{P}^{n+1})$.

Note that $\mathbb{C} = \mathbb{C}^* \cup {\infty}$.
Consider the morphism
$$
\Phi: C_{p+1,d}(\mathbb{P}^{n+1}) \times \mathbb{C}^* \to C_{p+1,d}(\mathbb{P}^{n+1}), \quad \Phi(c, t) = \phi_t(c).
$$
The Zariski closure of the graph of $\Phi$
in $(C_{p+1,d}(\mathbb{P}^{n+1}) \times \mathbb{C}) \times C_{p+1,d}(\mathbb{P}^{n+1})$,
denoted $\overline{\Gamma_{\Phi}}$,
is single-valued over
$T_{p+1,d}(\mathbb{P}^{n+1}) \times \mathbb{C} \subset C_{p+1,d}(\mathbb{P}^{n+1}) \times \mathbb{C}$
\cite[Rem. 4.6]{Lawson1}.
Consequently,
for an effective algebraic cycle $Z$ on $C_{p+1,d}(\mathbb{P}^{n+1})$
supported in $T_{p+1,d}(\mathbb{P}^{n+1})$,
the restriction $\overline{\Gamma_{\Phi}}|_{T_{p+1,d}(\mathbb{P}^{n+1}) \times \mathbb{C}}$
provides an algebraic equivalence between $\phi_1(Z) = Z$ and $\phi_{\infty}(Z)$,
with support in $\Sigma C_{p,d}(\mathbb{P}^n)$.

Let $f: \Lambda \to \mathcal{Z}_q(C_{p+1,d}(\mathbb{P}^{n+1}))$
be a continuous map such that the support of $f(\lambda)$
is contained in $T_{p+1,d}(\mathbb{P}^{n+1})$ for each $\lambda \in \Lambda$.
The automorphisms $\phi_t$ induce isomorphisms
$$
\tilde{\phi}_t: \mathcal{Z}_q(C_{p+1,d}(\mathbb{P}^{n+1})) \to \mathcal{Z}_q(C_{p+1,d}(\mathbb{P}^{n+1}))
$$
for each $t \in \mathbb{C}^*$.
For any cycle $Z$ on $C_{p+1,d}(\mathbb{P}^{n+1})$ with support in $T_{p+1,d}(\mathbb{P}^{n+1})$,
the limit $\tilde{\phi}_{\infty}(Z) := \lim_{t \to \infty} \phi_t(Z)$
exists as a cycle on $C_{p+1,d}(\mathbb{P}^{n+1})$
with support in $\Sigma C_{p,d}(\mathbb{P}^n)$.
These maps $\tilde{\phi}_t$, for all $t \in \mathbb{C}$, induce a homotopy between $f$ and $g$.
\end{proof}

\iffalse
\begin{remark}
For a projective algebraic variety $Z$  on $C_{p+1,d}(\mathbb{P}^{n+1})$
whose support is in $T_{p+1,d}(\mathbb{P}^{n+1})$,
we have the following morphisms
$$
Z\times\C^*\to T_{p+1,d}(\mathbb{P}^{n+1})\times \C^*\to C_{p+1,d}(\mathbb{P}^{n+1})\times \C^*\to C_{p+1,d}(\mathbb{P}^{n+1}).
$$
The morphism extends to a well-defined  algebraic map
(i.e., the graph is an algebraic subset in of the product) (see \cite[p.280]{Lawson1})
$$
\phi_{Z}:Z\times\C\to T_{p+1,d}(\mathbb{P}^{n+1})\times \C\to T_{p+1,d}(\mathbb{P}^{n+1}),
$$
whose image $W:=\phi_Z(Z\times\C)$,
where $\C=\C^*\cup \{\infty\}$.
The rational function field $K(W)$ of $W$ coincides with $K(Z\times \C^*)$.
Let $g\in K(W)$ be the element corresponding $t-1\in K(Z\times \C^*)$,
where $t$ is the local coordinate on $\C=\C^*\cup \{\infty\}$.
Then $div(g)=Z- \phi_{\infty}(Z)$.
\end{remark}
\fi

Given a continuous map $f: S^{k-2q} \to \mathcal{Z}_q(C_{p+1,d}(\mathbb{P}^{n+1}))$,
we investigate when $f$ is homotopic
to a continuous map $g: S^{k-2q} \to \mathcal{Z}_q(C_{p+1,d'}(\mathbb{P}^{n+1}))$
for large $d'$ such that
the support of $g(s)$ is contained in $T_{p+1,d'}(\mathbb{P}^{n+1})$
for each $s \in S^{k-2q}$,
via the sequence of embeddings in Equation \eqref{eqn2.1}.
Fix a linear embedding $\mathbb{P}^{n+1} \subset \mathbb{P}^{n+2}$
and two points $x_0, x_1 \in \mathbb{P}^{n+2} -\mathbb{P}^{n+1}$.
Each projection $p_i: \mathbb{P}^{n+2} - {x_i} \to \mathbb{P}^{n+1}$ ($i = 0, 1$)
defines an algebraic line bundle over $\mathbb{P}^{n+1}$.
Let $D \in C_{n+1,e}(\mathbb{P}^{n+2})$
be an effective divisor of degree $e$ in $\mathbb{P}^{n+2}$ not containing $x_0$ or $x_1$.
Denote by $\widetilde{\mathrm{Div}_e}(\mathbb{P}^{n+2}) \subset C_{n+1,e}(\mathbb{P}^{n+2})$
the subset of all such divisors $D$.

For any effective $(p+1)$-cycle $c \in C_{p+1,d}(\mathbb{P}^{n+1})$ of degree $d$,
we define a lift to a cycle with support in an effective divisor
$D \in \widetilde{\mathrm{Div}_e}(\mathbb{P}^{n+2})$ as
$$\Psi_D(c) = (\Sigma_{x_0} c) \cdot D,$$
where $\Sigma_{x_0}$ denotes the algebraic suspension with vertex $x_0$.
The map $\Psi(c, D) := \Psi_D(c)$ is continuous in $c$ and $D$, yielding a morphism
$$\Psi_D: C_{p+1,d}(\mathbb{P}^{n+1}) \to C_{p+1,de}(\mathbb{P}^{n+2} - \{x_0, x_1\}).$$
The composition with the projection $(p_0)_*$
satisfies $(p_0)_* \circ \Psi_D = e \cdot \text{id}$,
where $e \cdot c = c + \dots + c$ ($e$ times).
Composing $\Psi_D$ with the projection $(p_1)_*$ transforms cycles in $\mathbb{P}^{n+1}$
so that most intersect $\mathbb{P}^n$ properly. To see this,
consider the family of divisors $tD$, for $t \in \mathbb{C}$,
defined by scalar multiplication in the line bundle
$p_0: \mathbb{P}^{n+2} - \{x_0\} \to \mathbb{P}^{n+1}$.

Assume $x_1 \notin tD$ for all $t \in \mathbb{C}$.
The construction above yields a family of transformations
$$
F_{tD} := (p_1)* \circ \Psi_{tD}: C_{p+1,d}(\mathbb{P}^{n+1}) \to C_{p+1,de}(\mathbb{P}^{n+1}),
$$
for $t \in \mathbb{C}$, where $F_{0D}=e\cdot\text{id}$ (multiplication by $e$).
We investigate which divisors
$D \in \widetilde{\mathrm{Div}_e}(\mathbb{P}^{n+2})$ (i.e., $x_0 \notin D$
and $x_1 \notin \bigcup_{0 \leq t \leq 1} tD$) ensure that
$$F_{tD}(c) \in T_{p+1,de}(\mathbb{P}^{n+1})$$
for all $t \in \mathbb{C}^*$ and a fixed cycle $c \in C_{p+1,d}(\mathbb{P}^{n+1})$. Define
$$
B_c
:= { D \in \widetilde{\mathrm{Div}_e}(\mathbb{P}^{n+2}) \mid
F_{tD}(c) \notin T_{p+1,de}(\mathbb{P}^{n+1}) \text{ for some } t \in \mathbb{C}^* },
$$
i.e., the set of degree $e$ divisors in $\mathbb{P}^{n+2}$
such that some component of $(p_1)* \circ \Psi_{tD}(c)$
is contained in $\mathbb{P}^n$ for some $t \in \mathbb{C}^*$.
A key result used later is:
\begin{proposition}[\cite{Lawson1}, Lem.~5.11]\label{Prop3.4}.
For $c \in C_{p+1,d}(\mathbb{P}^{n+1})$,
the complex codimension of $B_c$ satisfies $\mathrm{codim}_{\mathbb{C}} B_c \geq \binom{p+e+1}{e}$.
\end{proposition}

For a manifold $Y$,
a map $f: Y \to \mathcal{Z}_q(C_{p+1,d}(\mathbb{P}^n))$ is \emph{regular}
if it is piecewise linear (PL) with respect to triangulations of $Y$
and $\mathcal{Z}_q(C_{p+1,d}(\mathbb{P}^n))$.
Since any map is homotopic to a regular one, we obtain the following result.

\begin{proposition}\label{Prop3.5}
Let $f:S^{k-2q}\to \mathcal{Z}_q(C_{p+1,d}(\mathbb{P}^{n+1}))$ be a continuous map.
For each integer $e \geq 0$ satisfying $2 \binom{p+e+1}{e} > k+1$,
there exists a divisor $D \in \widetilde{\mathrm{Div}_e}(\mathbb{P}^{n+2})$
such that $e \cdot f$ is homotopic to
a map $g: S^{k-2q} \to \mathcal{Z}_q(C_{p+1,de}(\mathbb{P}^{n+1}))$
with the support of $g(s)$ contained in $T_{p+1,de}(\mathbb{P}^{n+1})$
for each $s \in S^{k-2q}$.
\end{proposition}

\begin{proof}
For each $d, e>0$ and $p\leq n$,
there is a rational map
$$
F:C_{p+1,d}(\mathbb{P}^{n+1})\times \widetilde{Div_e}(\mathbb{P}^{n+2})\to C_{p+1,de}(\mathbb{P}^{n+1}),
(c, D)\mapsto F_{D}(c),
$$
given by intersection (see \cite[Prop.~3.4-3.5]{Friedlander1}).

Now for any  integer $e$ satisfying $ \binom{p+e+1}{e}>q+1$, define the subset
$$
B(f)=\bigcup_{s\in S^{k-2q},c\in f(s), 0\leq t\leq 1} t\cdot B_c,
$$
of $\widetilde{Div_e}(\mathbb{P}^{n+2})$,
where $ t\cdot B_c=\{tD|D\in B_c\}$.
By Proposition \ref{Prop3.4},
we have
$$
{\rm codim}_{\mathbb{R}} B(f)\geq  2\binom{p+e+1}{e}-k-1>0.
$$
Hence there exists a $D\in \widetilde{Div_e}(\mathbb{P}^{n+2})-B(f)$ such that
the map
$$
F_{tD}:C_{p+1,d}(\mathbb{P}^{n+1})\to C_{p+1,de}(\mathbb{P}^{n+1})
$$
induces an map
$$
\widetilde{F_{tD}}:\mathcal{Z}_{q}(C_{p+1,d}(\mathbb{P}^{n+1}))\to \mathcal{Z}_{q} (C_{p+1,de}(\mathbb{P}^{n+1}))
$$
satisfying $\widetilde{F_{tD}}\circ f$ is a homotopy between $e\cdot f$ and $\widetilde{F_{D}}\circ f$.
The support of $\widetilde{F_D} \circ f(s)$ is contained in $T_{p+1,de}(\mathbb{P}^{n+1})$.
Thus, we define $g = \widetilde{F_D} \circ f$, completing the proof of the proposition.
\end{proof}

The proof of Theorem \ref{Thm2.1} relies on the following result:
\begin{lemma}\label{lemma3.6}For any $d > 0$, $0 \leq p \leq n$, and $0 \leq 2q \leq k$,
the cycle class map
$$
cl_{\mathbb{Q}}:L_qH_k(\mathrm{SP}^d\mathbb{P}^{n})_{\mathbb{Q}}\to H_k(\mathrm{SP}^d\mathbb{P}^{n},\mathbb{Q})
$$
is an isomorphism.
\end{lemma}
\begin{proof}For a projective variety $X$ with a finite group $G$ action such that $Y = X/G$,
we have isomorphisms $L_qH_k(Y)_{\mathbb{Q}} \cong L_qH_k(X)_{\mathbb{Q}}^G$ \cite[Prop.~1.2]{Hu-Li}
and $H_k(Y, \mathbb{Q}) \cong H_k(X, \mathbb{Q})^G$.
If $L_qH_k(X)_{\mathbb{Q}} \cong H_k(X, \mathbb{Q})$,
the naturality of the cycle class map
from Lawson homology to singular homology implies $L_qH_k(Y)_{\mathbb{Q}} \cong H_k(Y, \mathbb{Q})$.
Applying this to $X = (\mathbb{P}^n)^d$ with the symmetric group $G = S_d$,
since $\mathrm{SP}^d \mathbb{P}^n = (\mathbb{P}^n)^d / S_d$
and $L_qH_k((\mathbb{P}^n)^d)_{\mathbb{Q}} \cong H_k((\mathbb{P}^n)^d, \mathbb{Q})$,
the lemma follows.
\end{proof}

%(in fact morphism from the construction of $F_{tD}$)

\begin{remark}
Lemma \ref{lemma3.6} is fundamental
to computing the Lawson homology groups of Chow varieties in rational coefficients.
Whether the cycle class map
$$
cl:L_qH_k(\mathrm{SP}^d\mathbb{P}^{n})\to H_k(\mathrm{SP}^d\mathbb{P}^{n},\mathbb{Z})
$$
is an isomorphism for all $d > 0$, $0 \leq p \leq n$, and $0 \leq 2q \leq k$ remains an open question,
as a special case of Conjecture \ref{conj1}.
This statement was once assumed to be ``trivially'' true until a proof was sought.
Proving or disproving it is an intriguing challenge and merits formulation as an independent conjecture.
\end{remark}

\iffalse
\begin{remark}
The Lawson homology group (in rational coefficients)
of infinite symmetric product of a smooth complete variety $X$
have been defined in \cite{Kimura-Vistoli}.
The relations between the Lawson homology group of infinite symmetric product of $X$
and that of a finite symmetric product of $X$ are given by stability conjectures.
It seems that the only computable Lawson homology groups of a finite symmetric product  is the case
that $X$ is a smooth projective curve.
\end{remark}
\fi

We need the following lemma to show Theorem \ref{Thm3.1}.
\begin{lemma}\label{lemma3.9}
For any $0\leq p'\leq n'$ and $0\leq 2q\leq k$,
we have the following commutative diagram
\begin{equation}\label{eqn3.1}
\vcenter{
\xymatrix{&L_qH_k(\mathcal{C}_{p'}(\mathbb{P}^{n'}))\ar[r]^-{cl}\ar@{->>}[d]^{\Sigma_*^{l}}&H_k(\mathcal{C}_{p'}(\mathbb{P}^{n'}),\mathbb{Z})\ar[d]^{\Sigma_*^h}_{\cong}\\
&L_qH_k(\mathcal{C}_{p'+1}(\mathbb{P}^{n'+1}))\ar[r]^-{cl}&H_k(\mathcal{C}_{p'+1}(\mathbb{P}^{n'+1}),\mathbb{Z}),
}
}
\end{equation}
where we denote the suspension map on Lawson homology groups by $\Sigma_*^l$
and on singular homology groups by $\Sigma_*^h$ to distinguish them.
Then the map
$$
\Sigma_*^l: L_qH_k(\mathcal{C}_{p'}(\mathbb{P}^{n'})) \to L_qH_k(\mathcal{C}_{p'+1}(\mathbb{P}^{n'+1}))
$$
is surjective for $0 \leq 2q \leq k$.
\end{lemma}
\begin{proof}
The following commutative diagram holds,
where the right vertical map is an isomorphism by the Lawson suspension theorem(see \cite{Lawson1}):
\begin{equation*}
\xymatrix{&L_qH_k(\mathcal{C}_{p'}(\mathbb{P}^{n'}))\ar[r]^-{cl}\ar[d]^{\Sigma_*^l}&H_k(\mathcal{C}_{p'}(\mathbb{P}^{n'}),\mathbb{Z})\ar[d]^{\Sigma_*^h}_{\cong}\\
&L_qH_k(\mathcal{C}_{p'+1}(\mathbb{P}^{n'+1}))\ar[r]^-{cl}&H_k(\mathcal{C}_{p'+1}(\mathbb{P}^{n'+1}),\mathbb{Z}).
}
\end{equation*}

%By definition of the Lawson homology group $L_qH_k(\mathcal{C}_{p'+1}(\mathbb{P}^{n'+1}))$,an element $[Z]\in L_qH_k(\mathcal{C}_{p'+1}(\mathbb{P}^{n'+1}))$ comes from the image of the class of an algebraic cycle $Z$on $L_qH_k(C_{p'+1,d}(\mathbb{P}^{n'+1}))$ for $d$ large. For any algebraic cycle, we can write $Z=Z^+-Z^-$, where both $Z^+$ and $Z^-$ are effective algebraic cycles. Hence it suffices to prove the lemma for the class of effective cycles.

Applying Proposition \ref{Prop3.4} to
the map $f: S^{k-2q} \to \mathcal{Z}_q(C_{p'+1,d}(\mathbb{P}^{n'+1}))$,
we obtain that $e \cdot f$ is homotopic to a map
$$g: S^{k-2q} \to \mathcal{Z}_q(C_{p'+1,de}(\mathbb{P}^{n'+1}))$$
with the support of $g(s)$ contained in $T_{p'+1,de}(\mathbb{P}^{n'+1})$ for each $s \in S^{k-2q}$
and sufficiently large positive integer $e$.
By Proposition \ref{Prop3.2}, $g$ is homotopic to a map
$$\tilde{g}: S^{k-2q} \to \mathcal{Z}_q(C_{p'+1,de}(\mathbb{P}^{n'+1}))$$
with the support of $\tilde{g}(s)$ contained in $\Sigma C_{p',de}(\mathbb{P}^{n'})$
for each $s \in S^{k-2q}$.
Hence $e \cdot f$ is homotopic to $\tilde{g}$. \

Since $[e \cdot f] = e [f]$ in homotopy groups,
we have $[f] = (e+1)[f] - e[f] = [\tilde{g}_1] - [\tilde{g}]$,
where $(e+1) \cdot f$ is homotopic to a map
$$
\tilde{g}_1: S^{k-2q} \to \mathcal{Z}_q(C_{p'+1,d(e+1)}(\mathbb{P}^{n'+1}))
$$
with the support of $\tilde{g}_1(s)$ contained in $\Sigma C_{p',d(e+1)}(\mathbb{P}^{n'})$
for each $s \in S^{k-2q}$.
Taking the limit as $d \to \infty$ for the suspension map
$$
\Sigma_*^l: L_qH_k(C_{p',d}(\mathbb{P}^{n'})) \to L_qH_k(C_{p'+1,d}(\mathbb{P}^{n'+1})),
$$
we deduce that
$$
\Sigma_*^l: L_qH_k(\mathcal{C}_{p'}(\mathbb{P}^{n'})) \to L_qH_k(\mathcal{C}_{p'+1}(\mathbb{P}^{n'+1}))
$$
is surjective, completing the proof of the lemma.
\end{proof}

Now we can prove Theorem \ref{Thm3.1} and Theorem \ref{Thm2.1}.

\begin{proof}[Proof of Proposition \ref{Thm3.1} and Theorem \ref{Thm2.1}]
For any projective variety $X$, there exists a natural cycle class map
$$cl: L_qH_k(X) \to H_k(X, \mathbb{Z}).$$
In particular, for Chow varieties, we have
$$cl: L_qH_k(C_{p,d}(\mathbb{P}^n)) \to H_k(C_{p,d}(\mathbb{P}^n), \mathbb{Z}).$$
When $p = 0$, Theorem \ref{Thm2.1} follows directly from Lemma \ref{lemma3.6}.
For $p > 0$, set $p = p' + 1$. The following commutative diagram holds,
where the right vertical map is an isomorphism by the Lawson suspension theorem:
\begin{equation*}
\xymatrix{
&L_qH_k(\mathcal{C}_{p'}(\mathbb{P}^{n'}))\ar[r]^-{cl}\ar@{->>}[d]^{\Sigma_*^l}&H_k(\mathcal{C}_{p'}(\mathbb{P}^{n'}),\mathbb{Z})\ar[d]^{\Sigma_*^h}_{\cong}\\
&L_qH_k(\mathcal{C}_{p'+1}(\mathbb{P}^{n'+1}))\ar[r]^-{cl}&H_k(\mathcal{C}_{p'+1}(\mathbb{P}^{n'+1}),\mathbb{Z})
}
\end{equation*}
for any $0\leq p'\leq n'$ and $k\geq 2q\geq 0$.

By Lemma \ref{lemma3.9},
the suspension map
$\Sigma_*^l: L_qH_k(\mathcal{C}_{p'}(\mathbb{P}^{n'})) \to L_qH_k(\mathcal{C}_{p'+1}(\mathbb{P}^{n'+1}))$
is surjective. From \cite[Thm.~3]{Lawson1},
the suspension map
$$\Sigma_*^h: H_k(\mathcal{C}_{p'}(\mathbb{P}^{n'},\mathbb{Z}) \to H_k(\mathcal{C}_{p'+1}(\mathbb{P}^{n'+1}),
\mathbb{Z})$$ is an isomorphism.
Thus,  from the commutative diagram in \eqref{eqn3.1},
we obtain that the upper horizontal map $cl$ is surjective
if and only if the lower horizontal map $cl$ is surjective.
Moreover, if the upper horizontal map $cl$ is injective,
then both $\Sigma_*^l$ and the lower horizontal map $cl$ are injective.

After tensoring with rational coefficients in Equation \eqref{eqn3.1},
we have
\begin{equation}\label{eqn3.2}
\vcenter{
\xymatrix{&L_qH_k(\mathcal{C}_{p'}(\mathbb{P}^{n'}))_{\mathbb{Q}}\ar[r]^-{cl_{\mathbb{Q}}}\ar@{->>}[d]^{\Sigma_*^l}
&H_k(\mathcal{C}_{p'}(\mathbb{P}^{n'}),\mathbb{Q})\ar[d]^{\Sigma_*^h}_{\cong}\\
&L_qH_k(\mathcal{C}_{p'+1}(\mathbb{P}^{n'+1}))_{\mathbb{Q}}\ar[r]^-{cl_{\mathbb{Q}}}&H_k(\mathcal{C}_{p'+1}(\mathbb{P}^{n'+1}),\mathbb{Q}).
}
}
\end{equation}

By Lemma \ref{lemma3.6},
the upper horizontal cycle class map
$cl_{\mathbb{Q}}:
L_qH_k(\mathcal{C}_{0}(\mathbb{P}^{n'}))_{\mathbb{Q}}
\to H_k(\mathcal{C}_{0}(\mathbb{P}^{n'}), \mathbb{Q})$
in the commutative diagram \eqref{eqn3.2} is an isomorphism for $p' = 0$.
By induction on $p'$,
both $\Sigma^l: L_qH_k(\mathcal{C}_{p'}(\mathbb{P}^{n'})) \to
L_qH_k(\mathcal{C}_{p'+1}(\mathbb{P}^{n'+1}))$
and
$cl_{\mathbb{Q}}: L_qH_k(\mathcal{C}_{p'}(\mathbb{P}^{n'}))_{\mathbb{Q}} \to
H_k(\mathcal{C}_{p'}(\mathbb{P}^{n'}), \mathbb{Q})$
are isomorphisms for all $0 \leq p' \leq n'$ and $0 \leq 2q \leq k$.
This completes the proof of Theorem \ref{Thm3.1} and Theorem \ref{Thm2.1}.
\end{proof}

\iffalse
\begin{remark}
It should be cautious about the fact that the Lawson homology group is NOT a homotopy invariant.
Therefore the Lawson homology group $L_qH_k(T_{p+1,d}(\mathbb{P}^{n+1}))$ is in general not isomorphic to
$L_qH_k(C_{p,d}(\mathbb{P}^{n}))$,
although $T_{p+1,d}(\mathbb{P}^{n+1})$ is homotopy equivalent to $C_{p,d}(\mathbb{P}^{n})$.
The Lawson homology group $L_qH_k(T_{p+1,d}(\mathbb{P}^{n+1}))$ never appears in the proof of our results.
As a comparison, such a homotopy equivalence is essential in Lawson's proof of the suspension theorem.

Actually, in the proof of Theorem \ref{Thm2.1},  we first show the surjectivity of
$$
\Sigma_*^l:L_qH_k(\mathcal{C}_{p}(\mathbb{P}^{n}))_{\mathbb{Q}}\to L_qH_k(\mathcal{C}_{p+1}(\mathbb{P}^{n+1}))_{\mathbb{Q}},
$$
and then we prove the injectivity of $\Sigma_*^l$ by induction on $p$
and the corresponding isomorphism of singular homology groups,
while the Lawson suspension theorem is indispensable.
\end{remark}
\fi

%%%%%%%%%%%%%%%%%%%%%%%%%##################333333333333333333333333333333333333333333333333
\section{Proof of Theorem \ref{Thm2.2} and Theorem \ref{Thm2.3}}

In the preceding section, we constructed the map
$$
F_{tD}:=(p_1)_* \circ \Psi_{tD} : C_{p+1,d}(\mathbb{P}^{n+1}) \to C_{p+1,de}(\mathbb{P}^{n+1}).
$$
Furthermore, when $ D \notin B_c $,
the image of $ F_{tD} $ lies in the Zariski open subset $ T_{p+1,de}(\mathbb{P}^{n+1}) $.
Such a $ D $ exists provided that the complex codimension of $ B_c $,
given by $ \mathrm{codim}_{\mathbb{C}} B_c \geq \binom{p+e+1}{e} $, is positive.

By Proposition \ref{Prop3.5},
for a continuous map $ f : S^{k-2q} \to \mathcal{Z}_q(C_{p+1,d}(\mathbb{P}^{n+1})) $,
the map $ e \cdot f $ is homotopic to a map
$$
g : S^{k-2q} \to \mathcal{Z}_q(C_{p+1,de}(\mathbb{P}^{n+1}))
$$
such that the support of $ g(s) $
is contained in $ T_{p+1,de}(\mathbb{P}^{n+1}) $ for each $ s \in S^{k-2q} $,
provided $ e $ is a sufficiently large positive integer.
Specifically, $ e \cdot f = \widetilde{F}_{0D}(f) $
is homotopic to the map $ g = \widetilde{F}_{D}(f) $.
Consequently, we obtain the following commutative diagram:
\begin{equation}\label{eqn4.1}
\begin{aligned}
\xymatrix{
& \mathcal{Z}_q(C_{p+1,d}(\mathbb{P}^{n+1})) \ar[r] \ar@{^(->}[d] &
\mathcal{Z}_q(C_{p+1,de}(\mathbb{P}^{n+1})) \ar@{^(->}[d] \\
S^{k-2q} \ar[r]^-{f} & \mathcal{Z}_q(C_{p+1,d}(\mathbb{P}^{n+1})) \ar[r]^-{\widetilde{F}_{tD}}
\ar[ur]^-{\widetilde{F}_{D}} & \mathcal{Z}_q(C_{p+1,de}(\mathbb{P}^{n+1})),
}
\end{aligned}
\end{equation}
where $ \widetilde{F}_{D}=\widetilde{F}_{1D} $.

\begin{lemma}\label{lemma4.1}
For any integer $ e \geq 1 $, the map
$$
\phi_e : C_{0,d}(\mathbb{P}^{n}) \to C_{0,de}(\mathbb{P}^{n}), \quad \phi_e(c) = e \cdot c
$$
induces isomorphisms
$$
\phi_{e*} : L_q H_k(C_{0,d}(\mathbb{P}^{n}))_{\mathbb{Q}} \to L_q
H_k(C_{0,de}(\mathbb{P}^{n}))_{\mathbb{Q}}
$$
for $ 2q \leq k \leq 2d $.
\end{lemma}
\begin{proof}
We first note that $ C_{0,d}(\mathbb{P}^n) \cong \mathrm{SP}^d(\mathbb{P}^n) $,
where $\mathrm{SP}^d(\mathbb{P}^n) $
denotes the $ d $-th symmetric product of $ \mathbb{P}^n $.
By Lemma \ref{lemma3.6}, we have
$$
L_q H_k(C_{0,d}(\mathbb{P}^n))_{\mathbb{Q}} \cong H_k(C_{0,d}(\mathbb{P}^n))_{\mathbb{Q}}
$$
for all $ d, n $, and $ k \geq 2q \geq 0$.
Furthermore, by Lemma~9 and Remark~10 in~\cite{Hu-2015},
we have
$$
H_k(C_{0,d}(\mathbb{P}^n))_{\mathbb{Q}} \cong H_k(C_{0,de}(\mathbb{P}^n))_{\mathbb{Q}}
$$
for $0\leq 2q \leq k \leq 2d $. This completes the proof of the lemma.
\end{proof}

\begin{remark}
We expect that the maps
$$
\phi_{e*} : L_q H_k(C_{0,d}(\mathbb{P}^{n})) \to L_q H_k(C_{0,de}(\mathbb{P}^{n}))
$$
are isomorphisms for $ 0\leq 2q \leq k \leq 2d $, under the assumption of Lemma~\ref{lemma4.1}.
\end{remark}

Note that there is a commutative diagram
\begin{equation}\label{eqn4.2}
\vcenter{
\xymatrix{
C_{p,d}(\mathbb{P}^{n}) \ar[r]^-{\Phi_{p,d,n,e}} \ar[d]^-{\Sigma}
        & C_{p,de}(\mathbb{P}^{n}) \ar[d]^-{\Sigma} \\
T_{p+1,d}(\mathbb{P}^{n+1}) \ar[r]^-{F_{tD}} & T_{p+1,de}(\mathbb{P}^{n+1}),
}}
\end{equation}
where $ \Phi_{p,d,n,e} $ (with $ \Phi_{0,d,n,e} = \phi_e $
in Lemma~\ref{lemma4.1}) is the composite map
$$
\begin{array}{ccccc}
C_{p,d}(\mathbb{P}^{n}) & \to & C_{p,d}(\mathbb{P}^{n}) \times \cdots \times C_{p,d}(\mathbb{P}^{n}) & \to & C_{p,de}(\mathbb{P}^{n}) \\
c & \mapsto & (c, \ldots, c) & \mapsto & e \cdot c
\end{array}
$$
and for $ D \in \widetilde{\mathrm{Div}}_e(\mathbb{P}^{n+2}) $,
the map $ F_{tD} $ in~\eqref{eqn4.2} is the restriction of
$$
F_{tD} : C_{p+1,d}(\mathbb{P}^{n+1}) \to C_{p+1,de}(\mathbb{P}^{n+1})
$$
to $ T_{p+1,d}(\mathbb{P}^{n+1}) $,
since its image under this restriction lies in $ T_{p+1,de}(\mathbb{P}^{n+1}) $
(see~\cite[Lem.~5.5]{Lawson1}).

\begin{proposition}\label{Prop4.3}
For integers $ p, d, n $ such that $ 0 \leq p \leq n $ and $ d \geq 0 $,
there exists an integer $ e_{p,d,n} \geq 1 $
such that for all $ e \geq e_{p,d,n} $,
the induced map
$$
(\Phi_{p+1,d,n+1,e})_* : L_q H_k(C_{p+1,d}(\mathbb{P}^{n+1}))_{\mathbb{Q}} \to L_q H_k(C_{p+1,de}(\mathbb{P}^{n+1}))_{\mathbb{Q}}
$$
on Lawson homology groups with rational coefficients, induced by
$$
\Phi_{p+1,d,n+1,e} : C_{p+1,d}(\mathbb{P}^{n+1}) \to C_{p+1,de}(\mathbb{P}^{n+1}), \quad c \mapsto e \cdot c,
$$
is an isomorphism for $ 0 \leq 2q \leq k \leq 2d $.
Furthermore, we have the following commutative diagram of isomorphisms:
\begin{equation}\label{eqn4.3}
\vcenter{
\xymatrix{
L_q H_k(C_{p+1,d}(\mathbb{P}^{n+1}))_{\mathbb{Q}} \ar[r]^-{cl_{\mathbb{Q}}}_-{\cong} \ar[d]^-{(\Phi_{p+1,d,n+1,e})_*}_-{\cong} & H_k(C_{p+1,d}(\mathbb{P}^{n+1}), \mathbb{Q}) \ar[d]^-{(\Phi_{p+1,d,n+1,e})_*}_-{\cong} \\
L_q H_k(C_{p+1,de}(\mathbb{P}^{n+1}))_{\mathbb{Q}} \ar[r]^-{cl_{\mathbb{Q}}}_-{\cong} & H_k(C_{p+1,de}(\mathbb{P}^{n+1}), \mathbb{Q})
}}
\end{equation}
for $ 0 \leq 2q \leq k \leq 2d $.
\end{proposition}

\begin{proof}
We prove the isomorphism property of $ (\Phi_{p+1,d,n+1,e})_* $ by induction on $ p $.
The base case $ p = -1 $ follows from Lemma~\ref{lemma4.1}.
Assume that the map
$$
\Phi_{p,d,n,e} : C_{p,d}(\mathbb{P}^{n}) \to C_{p,de}(\mathbb{P}^{n}),
\quad \Phi_{p,d,n,e}(c) = e \cdot c
$$
induces isomorphisms
$$
(\Phi_{p,d,n,e})_* : L_q H_k(C_{p,d}(\mathbb{P}^{n}))_{\mathbb{Q}}
\to L_q H_k(C_{p,de}(\mathbb{P}^{n}))_{\mathbb{Q}}
$$
for $ 0 \leq 2q \leq k \leq 2d $ and $ e \geq e_{p,d,n} $.
%By induction, both $\pi_k(C_{p,d}(\mathbb{P}^{n}))$ and $\pi_k(C_{p,de}(\mathbb{P}^{n}))$ are isomorphic to $\mathbb{Z}$ for $k\leq 2d$ and %$(\Phi_{p,d,n,e})_*:\pi_k(C_{p,d}(\mathbb{P}^{n}))\to \pi_k(C_{p,de}(\mathbb{P}^{n}))$ is  multiplication by $e:\mathbb{Z}\to \mathbb{Z}$.
\iffalse
Moreover, by induction we have the following commutative diagram
\begin{equation*}
\xymatrix{&\Ch_{q}(C_{p,d}(\mathbb{P}^{n}))_{\mathbb{Q}}\ar[r]^-{cl_{\mathbb{Q}}}_{\cong}\ar[d]^{(\Phi_{p,d,n,e})_*}_{\cong}
&H_k(C_{p,d}(\mathbb{P}^{n}),\mathbb{Q})\ar[d]^{(\Phi_{p,d,n,e})_*}_{\cong}\\
&L_qH_k(C_{p,de}(\mathbb{P}^{n}))_{\mathbb{Q}}\ar[r]^-{cl_{\mathbb{Q}}}_{\cong}&H_k(C_{p,de}(\mathbb{P}^{n}),\mathbb{Q})
}
\end{equation*}
for $0\leq 2q \leq k\leq 2d$ and $e\geq e_{p,d,n}$.
\fi

By Equations~\eqref{eqn4.1} and~\eqref{eqn4.2},
we have the following commutative diagram:
\begin{equation*}
\xymatrix{
& \mathcal{Z}_q(C_{p,d}(\mathbb{P}^{n})) \ar[r]^-{(\Phi_{p,d,n,e})_*} \ar[d]^-{\Sigma_*}
          & \mathcal{Z}_q(C_{p,de}(\mathbb{P}^{n})) \ar[d]^-{\Sigma_*} \\
& \mathcal{Z}_q(C_{p+1,d}(\mathbb{P}^{n+1})) \ar[r] \ar@{^(->}[d]
        & \mathcal{Z}_q(C_{p+1,de}(\mathbb{P}^{n+1})) \ar@{^(->}[d] \\
S^{k-2q} \ar[r]^-{f}
& \mathcal{Z}_q(C_{p+1,d}(\mathbb{P}^{n+1})) \ar[r]^-{\widetilde{F}_{0D}} \ar[ur]^-{\widetilde{F}_{D}}
              & \mathcal{Z}_q(C_{p+1,de}(\mathbb{P}^{n+1})).
}
\end{equation*}
Consequently, we obtain the following commutative diagram on Lawson homology groups:
\begin{equation}\label{eqn4.4}
\vcenter{
\xymatrix{
& L_q H_k(C_{p,d}(\mathbb{P}^{n}))_{\mathbb{Q}} \ar[r]^-{(\Phi_{p,d,n,e})_*} \ar[d]^-{\Sigma_*}
       & L_q H_k(C_{p,de}(\mathbb{P}^{n}))_{\mathbb{Q}} \ar[d]^-{\Sigma_*} \\
& L_q H_k(C_{p+1,d}(\mathbb{P}^{n+1}))_{\mathbb{Q}} \ar[r]^-{(F_{0D})_*}
                  & L_q H_k(C_{p+1,de}(\mathbb{P}^{n+1}))_{\mathbb{Q}}.
}}
\end{equation}
By induction, the upper horizontal map $ (\Phi_{p,d,n,e})_* $ is an isomorphism.
Moreover, by Theorem~\ref{Thm3.1}, the map
$$
\Sigma_* : L_q H_k(\mathcal{C}_{p}(\mathbb{P}^{n}))_{\mathbb{Q}}
\to L_q H_k(\mathcal{C}_{p+1}(\mathbb{P}^{n+1}))_{\mathbb{Q}}
$$
is an isomorphism for $ 0 \leq 2q \leq k \leq 2d $. Thus, the map
$$
\Sigma_* : L_q H_k(C_{p,de}(\mathbb{P}^{n}))_{\mathbb{Q}}
\to L_q H_k(C_{p+1,de}(\mathbb{P}^{n+1}))_{\mathbb{Q}}
$$
is an isomorphism for sufficiently large $ e $.
Hence, the left vertical map $ \Sigma_* $ is injective,
and the lower horizontal map $ (F_{0D})_* $ is surjective.

\iffalse
In fact, we can prove the surjective property of $(F_{0D})_*$ directly.
For $\alpha\in L_qH_k(C_{p+1,d}(\mathbb{P}^{n+1}))_{\mathbb{Q}}$,
it is linear combination of the class of (irreducible) subvariety of dimension $q$
on $C_{p+1,d}(\mathbb{P}^{n+1})$.
It is enough to consider that $\alpha$ is the cycle class of a subvariety $V$.
Let $\alpha\in L_qH_k(C_{p+1,d}(\mathbb{P}^{n+1}))_{\mathbb{Q}}$ be an element such that $\alpha=[V]$ is the cycle class
of a subvariety $V$ of dimension $q$ on $C_{p+1,d}(\mathbb{P}^{n+1})$.
Let $i_V:V\to C_{p+1,d}(\mathbb{P}^{n+1})$ be the inclusion.
To prove that $(F_{0D})_*$ is surjective,
we need to show that for large $e$,
there is a cycle $V'$ on $C_{p,de}(\mathbb{P}^{n})$
such that $\Sigma_*([V'])=(F_{0D})_*([V])\in L_qH_k(C_{p+1,de}(\mathbb{P}^{n+1}))$.
Indeed,
by Proposition \ref{Prop3.5},
$F_{0D}(V)$ is rationally equivalent to
$F_{D}(V)\subset T_{p+1,de}(\mathbb{P}^{n+1})$ on $C_{p+1,de}(\mathbb{P}^{n+1})$ for $e\geq e_{p+1,d,n+1}$.
Meanwhile,
$F_{D}(V)$  is rationally equivalent
to an algebraic cycle $\Sigma V'\subset \Sigma C_{p,de}(\mathbb{P}^{n})$ on $C_{p+1,de}(\mathbb{P}^{n+1})$.
\fi

Since
$e\cdot [f]=[e\cdot f]=[\widetilde{F}_{0D}(f)]
=(\widetilde{F}_{0D})_*([f])\in L_q H_k(C_{p+1,de}(\mathbb{P}^{n+1})) $,
we conclude that $ (\widetilde{F}_{0D})_* $
is injective for sufficiently large $ e $ (see Notes~5.1 in~\cite{Lawson1} for homotopy groups).
Therefore, for sufficiently large $ e $, the map
$$
(\Phi_{p+1,d,n+1,e})_* = (\widetilde{F}_{0D})_* : L_q H_k(C_{p+1,d}(\mathbb{P}^{n+1}))_{\mathbb{Q}} \to L_q H_k(C_{p+1,de}(\mathbb{P}^{n+1}))_{\mathbb{Q}}
$$
is an isomorphism for $ 0 \leq 2q \leq k \leq 2d $.

We have shown that the left vertical map in the following commutative diagram
\begin{equation*}
\xymatrix{
L_q H_k(C_{p+1,d}(\mathbb{P}^{n+1}))_{\mathbb{Q}} \ar[r]^-{cl_{\mathbb{Q}}}
\ar[d]^-{(\Phi_{p+1,d,n+1,e})_*}_-{\cong}
     & H_k(C_{p+1,d}(\mathbb{P}^{n+1}), \mathbb{Q}) \ar[d]^-{(\Phi_{p+1,d,n+1,e})_*}_-{\cong} \\
L_q H_k(C_{p+1,de}(\mathbb{P}^{n+1}))_{\mathbb{Q}} \ar[r]^-{cl_{\mathbb{Q}}}
         & H_k(C_{p+1,de}(\mathbb{P}^{n+1}), \mathbb{Q})
}
\end{equation*}
is an isomorphism for $ 0 \leq 2q \leq k \leq 2d $,
while the right vertical map is also an isomorphism (see~\cite{Hu-2015}).
As $ e \to \infty $, we obtain the following commutative diagram:
\begin{equation}\label{eqn4.5}
\vcenter{
\xymatrix{
L_q H_k(C_{p+1,d}(\mathbb{P}^{n+1}))_{\mathbb{Q}} \ar[r]^-{cl_{\mathbb{Q}}} \ar[d]_-{\cong}
          & H_k(C_{p+1,d}(\mathbb{P}^{n+1}), \mathbb{Q}) \ar[d]_-{\cong} \\
L_q H_k(\mathcal{C}_{p+1}(\mathbb{P}^{n+1}))_{\mathbb{Q}} \ar[r]^-{cl_{\mathbb{Q}}}
               & H_k(\mathcal{C}_{p+1}(\mathbb{P}^{n+1}), \mathbb{Q})
}}
\end{equation}
for $ 0 \leq 2q \leq k \leq 2d $.
By Theorem~\ref{Thm2.1}, the lower horizontal map in Equation~\eqref{eqn4.5} is an isomorphism.
Therefore, the upper horizontal map is an isomorphism for $ 0 \leq 2q \leq k \leq 2d $.
This completes the proof of the proposition.
\end{proof}

\begin{proof}[Proof of Theorem \ref{Thm2.2}]
By Proposition~\ref{Prop4.3}, we have isomorphisms
$$
(\Phi_{p+1,d,n+1,e})_* : L_q H_k(C_{p+1,d}(\mathbb{P}^{n+1}))_{\mathbb{Q}}
\to L_q H_k(C_{p+1,de}(\mathbb{P}^{n+1}))_{\mathbb{Q}}
$$
for $ 0 \leq 2q \leq k \leq 2d $, $ 0 \leq p \leq n $,
and sufficiently large $ e $. As $ e \to \infty $,
we obtain isomorphisms
$$
L_q H_k(C_{p+1,d}(\mathbb{P}^{n+1}))_{\mathbb{Q}} \to L_q H_k(\mathcal{C}_{p+1}(\mathbb{P}^{n+1}))_{\mathbb{Q}}
$$
for $ 0 \leq 2q \leq k \leq 2d $, $ 0 \leq p \leq n $.
Since the inclusion $i:C_{p+1,d}(\mathbb{P}^{n+1})\to C_{p+1,d+1}(\mathbb{P}^{n+1})$
commutes with $ \Phi_{p+1,d,n+1,e} $,
we obtain the following commutative diagram of Lawson homology groups:
\begin{equation*}
\xymatrix{
L_q H_k(C_{p,d}(\mathbb{P}^{n}))_{\mathbb{Q}} \ar[r]^-{\cong} \ar[d]^-{i_*}
        & L_q H_k(\mathcal{C}_{p}(\mathbb{P}^{n}))_{\mathbb{Q}} \\
L_q H_k(C_{p,d+1}(\mathbb{P}^{n}))_{\mathbb{Q}} \ar[ru]^-{\cong}
            &
}
\end{equation*}
for $ 0 \leq 2q \leq k \leq 2d $ and $ 1 \leq p \leq n $.
This shows that the maps
$$
i_* : L_q H_k(C_{p,d}(\mathbb{P}^{n}))_{\mathbb{Q}}
     \to L_q H_k(C_{p,d+1}(\mathbb{P}^{n}))_{\mathbb{Q}}
$$
are isomorphisms for $ 0 \leq 2q \leq k \leq 2d $ and $ 1 \leq p \leq n $.

For $ p = 0 $, the map
$$
i_* : L_q H_k(C_{0,d}(\mathbb{P}^{n}))_{\mathbb{Q}} \to L_q H_k(C_{0,d+1}(\mathbb{P}^{n}))_{\mathbb{Q}}
$$
is an isomorphism for $ 0 \leq 2q \leq k \leq 2d $,
since, by Lemma~\ref{lemma3.6},
the cycle class map is a natural isomorphism and the map
$$
i_* : H_k(C_{0,d}(\mathbb{P}^{n}))_{\mathbb{Q}} \to H_k(C_{0,d+1}(\mathbb{P}^{n}), \mathbb{Q})
$$
is also an isomorphism for $ 0 \leq 2q \leq k \leq 2d $.
This completes the proof of the theorem.
\end{proof}

\begin{proof}[Proof of Theorem \ref{Thm2.3}]
In the commutative diagram~\eqref{eqn4.4},
we conclude that the left vertical map
$$
\Sigma_* : L_q H_k(C_{p,d}(\mathbb{P}^{n}))_{\mathbb{Q}}
\to L_q H_k(C_{p+1,d}(\mathbb{P}^{n+1}))_{\mathbb{Q}}
$$
is an isomorphism for $ 0 \leq 2q \leq k \leq 2d $,
since all other three maps are isomorphisms, as shown in the proof of Proposition~\ref{Prop4.3}.
This completes the proof of Theorem~\ref{Thm2.3}.
\end{proof}

\begin{proof}[Proof of Corollary \ref{cor2.4} ]

It follows immediately
from the upper horizontal isomorphism in Equation~\eqref{eqn4.3} of Proposition~\ref{Prop4.3}.
\end{proof}

%%%%%%%%%%%%%%%%%%%%%%%%%%%%%%%%%%%%%%%%%%%%%%%%%%%%%%%%%%%%%%%%%%%%%
%%%%%
%%%%%%%%%%%%%%%%%%%%%%%%%%%%%%%%%%%%%%%%%%%%%%%%%%%%%%%%%%%%%%%%%%%%%

\section{Application to cycles of fixed degree}
In this section, we apply our result to the space of all algebraic cycles of a fixed degree.
Let $d\geq 1$ be a fixed integer.
Consider the spaces
\begin{equation}\label{eqn5.1}
\mathcal{D}(d):=\lim_{p,q\to \infty} C_{p,d}(\mathbb{P}^{p+q})
\end{equation}
of cycles of a fixed degree (with arbitrary dimension and codimension),
as introduced in \cite{Lawson-Michelsohn},
where  the limit for $p$ is given by suspension $\Sigma:C_{p,d}(\mathbb{P}^n)\to C_{p+1,d}(\mathbb{P}^{n+1}) $
and the limit for $q$ is induced by the inclusion $\mathbb{P}^{p+q}\subset \mathbb{P}^{p+q+1}$,
i.e. a $p$-cycle in $C_{p,d}(\mathbb{P}^{p+q})$ is viewed as a $p$-cycle in $C_{p,d}(\mathbb{P}^{p+q+1})$.

Then there is a filtration (see \cite[\S7]{Lawson-Michelsohn})
$$
BU=\mathcal{D}(1)\subset\mathcal{D}(2)\subset\cdots\subset \mathcal{D}(\infty),
$$
where $BU=\lim\limits_{q\to \infty}BU_q$ and
$\mathcal{D}(\infty)$ is homotopic to $K(\mathbb{Z},even)=\prod_{i=1}^{\infty} K(\mathbb{Z},2i)$
 (the weak product of Eilenberg--Maclane spaces).

The inclusion map $\mathcal{D}(d)\subset \mathcal{D}(\infty)$ induces homomorphisms on both homology and homotopy groups. Lawson and Michelsohn proved that the homomorphism on homotopy groups induced by $\mathcal{D}(1)\subset \mathcal{D}(\infty)$ is injective, and as abstract groups, $\pi_*(\mathcal{D}(1))\cong \pi_*(\mathcal{D}(\infty))$. They then posed the question of whether the induced homomorphism
$$
\pi_k(\mathcal{D}(d))\to \pi_k(\mathcal{D}(\infty))
$$
is always injective for all $k\geq 0$ and $d\geq 1$ (see \cite{Lawson-Michelsohn}).

This section addresses a Lawson homology group version of a question originally posed by Lawson and Michelsohn.
The inclusion map $\mathcal{D}(d)\subset \mathcal{D}(\infty)$ naturally induces maps on Lawson homology groups,
where the Lawson homology group of $\mathcal{D}(d)$ is defined as the direct limit of the Lawson homology groups of the Chow varieties:
$$
L_qH_k(\mathcal{D}(d)):=\lim\limits_{p,n\to \infty}L_qH_k(C_{p,d}(\mathbb{P}^n)).
$$
Note that for the singular homology group of $\mathcal{D}(d)$, there exist isomorphisms
$$
H_{m}(\mathcal{D}(d)):=\lim_{p,n\to \infty}H_{m}(C_{p,d}(\mathbb{P}^n))
$$
for all $1\leq d\leq \infty$.

\begin{question}\label{ques1}
Is the induced homomorphism $L_qH_k(\mathcal{D}(d))\to L_qH_k(\mathcal{D}(\infty))$ injective for all $d\geq 1$ and $k\geq 2q\geq 0$?
\end{question}

As a corollary of Theorem \ref{Thm2.2},
for $d=1$, the answer to Question \ref{ques1} is affirmative. 
More precisely, we have the following results.
\begin{proposition}
The induced homomorphism $L_qH_k(\mathcal{D}(1))\to L_qH_k(\mathcal{D}(\infty))$
by the inclusion map $\mathcal{D}(1)\subset \mathcal{D}(\infty)$ is an injective map,
and is an surjective map with rational coefficients.
\end{proposition}

The homotopic version of the above question was proposed in \cite{Lawson-Michelsohn},
which was answered negatively in \cite{Hu-Deg2} for $d=2$
through an explicit calculation of homology groups of $\mathcal{D}(2)$.
It is too optimistic to get a positive answer to Question \ref{ques1}.

However, we have the following result.
\begin{theorem}
The induced map on the Lawson homology groups
$$
i_*:L_qH_k(\mathcal{D}(d))_{\mathbb{Q}}\to L_qH_k(\mathcal{D}(\infty))_{\mathbb{Q}}
$$
from the inclusion
$i:\mathcal{D}(d)\subset \mathcal{D}(\infty)$
is an isomorphism for $2q\leq k\leq 2d$.
\end{theorem}
\begin{proof}
The result follows from Theorem \ref{Thm2.2} by taking the limit
$p,n\to \infty$.
\end{proof}

\end{document}